\documentclass[times, 10pt,twocolumn]{article} 
\usepackage{latex8}
\usepackage{times}

\usepackage{amssymb,amsfonts,amsmath,amsthm,amscd,bbold,stackrel,dsfont,mathrsfs}
\DeclareMathAlphabet{\mathpzc}{OT1}{pzc}{m}{it}

\newtheorem{propo}{Proposition}[section]
\newtheorem{lemma}[propo]{Lemma}
\newtheorem{definition}[propo]{Definition}

\newtheorem{thm}[propo]{Theorem}

\def\oX{\overline{X}}
\def\Rep{{\cal R}}
\def\Typ{{\sf Typ}}
\def\cA{{\cal A}}

\def\ogamma{\overline{\gamma}}

\newcommand{\reals}{{\mathds R}}

\newcommand{\naturals}{{\mathds N}}
\newcommand{\eqnsection}{\renewcommand{\theequation}{\thesection.\arabic{equation}}
      \makeatletter \csname @addtoreset\endcsname{equation}{section}\makeatother}
\def\eps{\epsilon}
\def\cC{{\cal C}}
\def\l|{\left|\left|}
\def\r|{\right|\right|}

\def\E{\mathds E}
\def\1{\mathds 1}
\def\prob{{\mathds P}}

\def\ind{{\mathds I}}

\def\ve{\varepsilon}

\def\cS{{\cal S}}

\def\hp{\widehat{p}}

\def\Zh{\widehat{Z}}

\def\onu{\overline{\nu}}
\def\de{{\rm d}}
\def\reals{{\mathds R}}

\def\ux{x}
\def\uX{X}

\def\Tree{{\sf T}}
\def\Ball{{\sf B}}
\def\Edge{{\sf D}}
\def\Graph{{\sf G}}
\def\cBall{\overline{\sf B}}

\def\sTV{\mbox{\tiny\rm TV}}

\def\cX{{\cal X}}

\def\sQ{{\sf Q}}

\begin{document}

\title{
Reconstruction  for models on random graphs}

\author{Antoine Gerschenfeld\\
Ecole Normale Sup\'erieure\\
45, rue d'Ulm 75005 Paris, France\\
 gerschen@clipper.ens.fr\\
\and
Andrea Montanari\\
Departments of Electrical Engineering and Statistics\\
Stanford University, Stanford CA-9305 USA\\
montanari@stanford.edu}

\maketitle
\thispagestyle{empty}

\begin{abstract}
Consider a collection of random variables attached to the vertices of 
a graph.
The reconstruction problem requires to estimate one of them
given `far away' observations. Several theoretical results (and simple 
algorithms) are available when their joint 
probability distribution is Markov with respect  to a tree.
In this paper we consider the case of
sequences of random graphs that converge locally to trees.
In particular, we develop a sufficient condition for the 
tree and graph reconstruction problem to coincide.
We apply such condition to colorings of random graphs.

Further, we characterize the behavior of Ising models 
on such graphs, both with attractive and random interactions 
(respectively, `ferromagnetic' and `spin glass').
\end{abstract}

\vspace{-0.5cm}

%
%
\Section{Introduction and outline}

Let $G=(V,E)$ be a graph, and $X=\{X_i:\, i\in V\}$ a 
proper coloring of its vertices sampled uniformly at random.
The \emph{reconstruction problem} amounts to estimating 
the color of a distinguished (root) vertex $r\in V$, 
when the colors $\{X_j=x_j:\, j\in U\}$ of subset of
vertices are revealed.
In particular, we want to understand whether the revealed 
values induce a substantial bias on the distribution of
$X_i$.

We shall consider the more general setting of
\emph{graphical models}. Such a model is defined
by a graph $G = (V,E)$, and a set of weights
$\psi = \{\psi_{ij}:\; (ij)\in E\}$,  
$\psi_{ij}:\cX\times\cX\to\reals_+$. Given a 
graph-weights pair $(G,\psi)$, we let 
\begin{eqnarray}
\prob\big\{\uX=\ux\big|(G,\psi)\big\}\equiv 
\frac{1}{Z}\, \prod_{(ij)\in E}\psi_{ij}(x_i,x_j)\, ,
\label{eq:GeneralGraphModel}
\end{eqnarray}
where we assume $\psi_{ij}(x,y)= \psi_{ij}(y,x)$.
The example of proper colorings is recovered by letting
$\cX = \{1,\dots,q\}$ ($q$ being the number of colors)
and $\psi_{ij}(x,y)=1$ if $x\neq y$ and $=0$ otherwise.
Ising models from statistical mechanics provide another interesting class,
whereby $\cX = \{+1,-1\}$. In the `ferromagnetic' case the
weights are $\psi_{ij}(+,+) = \psi_{ij}(-,-)=1-\eps$ and
$\psi_{ij}(+,-)=\psi_{ij}(-,+)=\eps$ for some $\eps\in [0,1/2]$.

For economy of notation, we shall often
write $\prob\{\,\cdot\,|G\}$ as a shorthand for  $\prob\{\,\cdot\,|(G,\psi)
\}$, and `the graph $G$' for `the graph-weights pair $(G,\psi)$.'
It is understood that, 
whenever $G$ is given, the weights $\psi$ are given as well.
Further, for $U\subseteq V_N$, 
we let $\uX_U = \{X_j:\, j\in U\}$ and 
$\prob_U\{\ux_U|G\}= \prob\{\uX_U=\ux_U|G\}$ be its marginal distribution
that can be obtained by marginalizing Eq.~(\ref{eq:GeneralGraphModel}).

For $i,j\in V$, let $d(i,j)$ be their graph theoretic
distance. Further for any $t\ge 0$, we let 
$\cBall(i,t)$ be the set of vertices $j$ such that $d(i,j)\ge t$,
(and, by abuse of notation, the induced subgraph).
The reconstructibility question asks whether the `far away'
variables $\uX_{\cBall(r,t)}$ 
provide significant information about $X_r$. 
This is captured by the following definition
(recall that, given two distributions $p$, $q$ on the same space $\cS$,
their total variation distance is 
$||p-q||_{\sTV}\equiv(1/2)\sum_{x\in \cS} |p_x-q_x|$).
\begin{definition}
The reconstruction problem is $(t,\ve)$-\emph{solvable} 
(\emph{reconstructible}) for the graphical model $(G,\psi)$
rooted at $r\in V$ if  
\begin{align*}
\Vert \prob_{r,\cBall(r,t)}\{\,\cdot\, ,\,\cdot\,|G\}-
 \prob_{r}\{\,\cdot\,|G\} \prob_{\cBall(r,t)}\{\,\cdot\,|G\}\Vert_{\sTV} \ge\ve\, .
\end{align*}
\end{definition}

In the following we will consider graphs 
$G$ that are themselves random. 
By this we mean that we will specify a joint distribution of the graph 
$G_N = (V_N=[N],E_N)$, 
of the weights $\{\psi_{ij}\}$, and of the root vertex 
$r$ whose variable we are interested in reconstructing. 
Equation (\ref{eq:GeneralGraphModel}) then specifies the conditional
distribution of $\uX$, given the random structure $(G_N,\psi)$
(again, we'll drop reference to $\psi$).
\begin{definition}
The reconstruction problem is \emph{solvable} 
(\emph{reconstructible}) for the 
sequence of random graphical models $\{G_N\}$ if there exists 
$\ve>0$ such that, for all $t\ge 0$ it is $(t,\ve)$-solvable 
with positive probability, i.e. if
\begin{align}
\Vert \prob_{r,\cBall(r,t)}\{\,\cdot\, ,\,\cdot\,|G_N\}-
 \prob_{r}\{\,\cdot\,|G_N\} \prob_{\cBall(r,t)}\{\,\cdot\,|G_N\}\Vert_{\sTV} \ge\ve\, .
\label{eq:GraphReconstruction}
\end{align}
with positive probability\footnote{Here and below, we say that 
the sequence of events $\{A_N\}$ holds with positive probability 
(wpp) if there exists $\delta>0$ and an infinite sequence
${\cal N}\subseteq \naturals$, such that $\prob\{A_N\}\ge \delta$
for any $N\in {\cal N}$. Notice that, in a random graph, $r$
might be in a small connected component. Therefore
Eq.~(\ref{eq:GraphReconstruction}) cannot be required  to hold
with high probability.}.
\end{definition}

To be specific, we shall assume $G_N$ to be a sparse random graph.
In this case, any finite neighborhood 
of $r$ converges in distribution to a tree \cite{Aldous}.
Further, imagine to mark the boundary vertices of such a neighborhood,
and then take the neighborhood out of $G_N$
(thus obtaining the subgraph denoted above as $\cBall(r,t)$).
The marked vertices will be (with high probability) `far apart'
from each other in $\cBall(r,t)$. This suggests that the corresponding 
random variables $\{X_j\}$ will be approximately independent when the 
tree-like neighborhood is taken out. 
Hence, approximating $G_N$ by its local tree structure might be a
good way to determine correlations between $X_r$ and the boundary
variables $\{X_{j}: d(r,j)=t\}$.
In other words, one would expect reconstructibility on $G_N$ 
to be determined by reconstructibility on the associated random tree.

Of course the above conclusion does not hold in general,
as it is based on a circular argument. We assumed that 
`far apart' variables (with respect to the residual graph $\cBall(r,t)$)
are weakly correlated, to understand whether `far apart' variable (in $G_N$)
are.
In fact, we will prove that tree and graph reconstruction do not coincide
in the simplest example one can think of, namely the Ising ferromagnet
(binary variables with attractive interactions).

On the positive side, we prove a general 
sufficient condition for the tree and graph behaviors to coincide.
The condition has a suggestive geometrical interpretation,
as it requires two independent random configurations $X^{(1)}$
and $X^{(2)}$ to be, with high probability, at an approximately fixed
`distance' from each other. In the example of coloring, we require
two uniformly random independent colorings of the same graph
to take the same value on about $1/q$ of the vertices.
The set of `typical configurations' looks like a sphere 
when regarded from any typical configuration.
Under such a condition, the above argument can be put on firmer basis.
We show that, once the the neighborhood of the root $r$ is taken out,
boundary variables become roughly independent. This in turns implies
that graph and tree reconstruction do coincide. 

We apply this sufficient condition to the Ising 
spin glass (where the condition 
can be shown to hold as a consequence of a recent 
result by Guerra and Toninelli \cite{Guerra}), and to
antiferromagnetic colorings of random graphs
(building on the work of Achlioptas and Naor \cite{AchlioptasNaor}). 
In both cases we will introduce a family of graphical models parametrized
by their
average degree. It is natural to expect reconstructibility to 
hold at large degrees (as the graph is `more connected')
and not to hold at small average degrees (since the graph `falls' apart into 
disconnected components). In the spin glass case we are indeed able to 
estabilish a threshold behavior (i.e. a critical degree value above 
which reconstruction is solvable). While we didn't achieve the same 
for colorings, we essentially reduced the problem to establishing a 
threshold for the tree model.
%
%
\SubSection{Applications and related work}

Let us discuss a selection of related problems that are  
relevant to our work.

\vspace{0.05cm}

{\bf Markov Chain Monte Carlo} (MCMC) algorithms provide a well established
way of approximating marginals of the distribution 
(\ref{eq:GeneralGraphModel}).
If the chain is reversible and has local updates, the mixing time is 
known to be related to the correlation decay properties
of the stationary distribution $\prob\{\,\cdot\,|G_N\}$ 
\cite{StaticDynamics1,StaticDynamics2}. In this context,
correlations between $X_r$ and $\uX_{\cBall(r,t)}$ are usually 
characterized by measures of the type
$\Delta(t)
\equiv\sup_{\ux}\Vert \prob_{r|\cBall(r,t)}\{\,\cdot\, |\ux_{\cBall(r,t)},G_N\}- \prob_{r}\{\,\cdot\,|G_N\} \Vert_{\sTV}$.
The `uniqueness ' condition requires $\Delta(t)$ to decay at large $t$,
and is easily shown to imply non-reconstructibility.
On graphs with sub-exponential growth, a fast enough decay is
a necessary and sufficient condition for fast mixing.
On the other hand, in more general cases this is too strong a criterion,
and one might want to replace it with the non-reconstructibility one.

In \cite{BergerEtAl} it was proved that non-reconstructibility
is equivalent to polynomial spectral gap for a class of models on trees.
The equivalence was sharpened in \cite{Martinelli}, showing that
non-reconstructibility is equivalent to fast mixing in the same models.
Further, \cite{BergerEtAl} proved that 
non-reconstructibility is a necessary condition for 
fast mixing on general graphs. While a converse does not hold in general,
non-reconstructibility is sufficient 
for rapid decay of the variance of local functions (which is often
regarded as the criterion for fast dynamics in physics) \cite{MonSem}.

\vspace{0.05cm}

{\bf Random constraint satisfaction problems.} 
Given an instance of a constraint satisfaction problem (CSP),
consider the uniform distribution over its solutions. This takes the form 
(\ref{eq:GeneralGraphModel}), where $\psi_{ij}$ is the indicator function 
over the constraint involving variables $x_i$, $x_j$ being satisfied
(Eq.~(\ref{eq:GeneralGraphModel}) is trivially generalized to
$k$-variables constraints).
For instance, in coloring it is the indicator function on
$x_i\neq x_j$. 

Computing the marginal $\prob_r\{\,\cdot\,|G_N\}$ can be useful both for
finding and for counting the solutions of such a CSP.
Assume to be able to generate \emph{one} uniformly random solution 
$\uX$. In general, this is not sufficient to approximate the
marginal of $X_i$ in any meaningful way. 
However one can try the following: fix all the variables 
`far from $r$' to take the same value as in the sampled configuration,
namely $X_{\cBall(r,t)}$, and compute the conditional distribution at the
root.
If the graph is locally tree-like, the conditional distribution of $X_r$
can be computed through an efficient dynamic programming procedure.
The result of this computation needs not to be near 
the actual marginal. However, non-reconstructibility implies 
the result to be with high probability within $\ve$ 
(in total variation distance) from the marginal.

As a consequence, a single sample (a single random solution $\ux$) 
is sufficient to approximate the marginal $\prob_r\{\,\cdot\,|G_N\}$.
The situation is even simpler under the sufficient condition in 
our main theorem (Theorem \ref{thm:graph<->tree}). In fact this implies that
the boundary condition $\ux_{\cBall(r,t)}$ can be replaced
by an iid uniform boundary.

For random  CSP's, $G_N$ becomes a sparse random graph.
Statistical mechanics studies \cite{MarcGiorgioRiccardo} 
suggest that, for typical instances  the set of solutions decomposes 
into `clusters' at sufficiently  large constraint density
\cite{Mora,AchlioRicci}.
This leads to the speculation that sampling from the uniform measure 
$\prob\{\,\cdot\,|G_N\}$ becomes harder in this regime.

The decomposition in clusters is related to reconstructibility, as per the
following heuristic argument.
Assume the set of solutions to be splitted into clusters, 
and that two solutions whose Hamming distance is smaller than $N\ve$
belong to the same cluster.
Then knowing the far away variables $\ux_{\cBall(r,t)}$ (i.e. all but a bounded
number of variables) does determine the cluster.
This in turns provides some information on $X_r$.

In fact, it was conjectured in \cite{MezardMontanari} 
that tree and graph reconstruction thresholds should coincide for
`frustrated' models on  random graphs. Both should
coincide with the clustering 
phase transition in the set of solutions \cite{OurPNAS}.

\vspace{0.05cm}

{\bf Statistical inference and message passing}.
Graphical models of the form  (\ref{eq:GeneralGraphModel}) 
are used in a number of contexts, from image processing
to artificial intelligence, etc. 
Statistical inference requires to compute marginals 
of such a distribution and message passing algorithms (in particular, 
belief  propagation, BP) are the methods of choice for accomplishing this 
task. 

The (unproven) assumption in such algorithms is that, if a tree neighborhood of
vertex 
$i$ is cut away from $G_N$, then the variables $\{X_j\}$ on the boundary 
of this tree are approximately independent. Assuming the marginals
of the boundary variables to be known, the marginal of $X_i$
can be computed through dynamic programming. Of course the marginals
to start from are unknown. However, the dynamic programming procedure
defines an mapping on the marginals themselves. In BP this mapping
is iterated recursively over all the nodes,  without convergence 
guarantees.

Lemma \ref{degenerescence} shows that, under the
stated conditions, the required independence condition does indeed
hold. As stressed above, this is instrumental in proving equivalence 
of graph and tree reconstructibility in Theorem \ref{thm:graph<->tree}.

The connection with message passing algorithm is further 
explored in \cite{DemboMontanari}. Roughly speaking that paper proves that,
if the reconstruction problem is unsolvable, than BP admits an approximate
fixed point that allows to compute the correct marginals.

\vspace{0.05cm}

Reconstruction problems also emerge in a variety of other contexts:
$(i)$ Phylogeny \cite{Phylogeny}
(given some evolved genomes,
one aims at reconstructing the genome of their common ancestor); 
$(ii)$ Network tomography \cite{Tomo}
(given end-to-end delays in a computer network,
infer the link delays in its interior); 
$(iii)$ Statistical mechanics \cite{Georgii,BleherRuizZagrebnov} 
(reconstruction  being related to the extremality of Gibbs measures).
%
%
\SubSection{Previous results}

If the graph $G_N$ 
is a tree, the reconstruction problem is relatively well understood
\cite{TreeRec}.
The fundamental reason is that the distribution 
$\prob\{\uX=\ux|G_N\}$ admits a simple description.
First sample the root variable $X_r$ from its marginal $\prob\{X_r=x_r|G_N\}$.
Then recursively for each node $j$, sample its children $\{X_l\}$
independently conditional on their parent value. 

Because of this Markov structure, one can prove a recursive
distributional equation for the conditional marginal at the root
$\prob_{r|\cBall(r,t)}\{\,\cdot\,|\uX_{\cBall(r,t)},G_N\}\equiv
\eta_t(\,\cdot\,)$
given the variable values at generation $t$. Notice that this is
a random quantity even for a deterministic graph $G_N$,
because $\uX_{\cBall(r,t)}$ is itself drawn randomly from the distribution 
$\prob\{\,\cdot\,|G_N\}$. Further, it contains all the information
(it is a `sufficient statistic') in the boundary
about the root variable $X_r$. 
In fact asymptotic behavior of $\eta_t(\,\cdot\,)$
as $t\to\infty$ then determines the solvability 
of the reconstruction problem.
Studying the asymptotic behavior of the sequence 
$\eta_t(\,\cdot\,)$ (which satisfies a recursive distributional 
equation) is the standard approach to tree reconstruction.

Among the other results, reconstructibility has been thoroughly characterized
for Ising models on generic trees 
\cite{BleherRuizZagrebnov,EvaKenPerSch,Asymmetric}. 
For an infinite tree $\Tree$ the reconstruction problem is solvable 
if and only if br$(\Tree)(1-2\eps)^2>1$,
whereby (for the cases treated below) br$(\Tree)$  coincides with the mean 
descendant number of any vertex. This result establishes a sharp threshold 
in the tree average degree (or in the parameter $\eps$), 
that we shall generalize to random graphs below. However, as we
will see, the  behavior is richer than in the tree case.

Reconstruction on general graphs poses new challenges,
since the above recursive description of
the measure $\prob\{\,\cdot\,|G_N\}$ is lacking.  
The result of \cite{BergerEtAl} allows to deduce 
non-reconstructibility from fast mixing of reversible MCMC with local 
updates. However, proving fast mixing is far from an easy task.
Further, the converse does not usually hold (one can have slow mixing
and non-reconstructibility).

An exception is provided by the recent paper by Mossel, Weitz and 
Wormald \cite{MWW} that
establishes a threshold for fast mixing for weighted independent sets
on random bipartite graphs (the threshold being in the weight parameter 
$\lambda$). Arguing as in Section 
\ref{sec:IsingFerro}, it can be shown that this is also the graph 
reconstruction threshold. This result is analogous to ours for the 
ferromagnetic Ising model: it provides an example in which the graph 
reconstruction threshold does not coincide with the 
tree reconstruction threshold. 
In both cases the graph reconstruction threshold coincides
instead with the tree `uniqueness threshold' 
(i.e. the critical parameter 
for the uniqueness condition mentioned above to hold).
%
%
\SubSection{Basic definitions}

We  consider two families of random graphical models:
\emph{regular} and \emph{Poisson} models. In both
cases the root $r\in V$ is uniformly 
random and independent of $G_N$.
A \emph{regular ensemble} is specified by assigning an alphabet 
$\cX$ (the variable range), a degree $(k+1)$ and and edge weight 
$\psi:\cX\times\cX\to\reals_+$. For any $N>0$, 
a random model is defined by letting $G_N$ be a uniformly random regular
graph of degree $(k+1)$ over vertex set  $V=[N]$. The joint distribution
of $(X_1,\dots,X_N)$ is given by Eq.~(\ref{eq:GeneralGraphModel}),
with $\psi_{ij}(\,\cdot\, ,\,\cdot\,)=\psi(\,\cdot\, ,\,\cdot\,)$.
A variation of this ensemble is obtained by letting 
$G$ be a random regular \emph{multi-graph} according to the configuration
model \cite{Bollobas}
(notice that our definitions make sense for multigraphs as well).
Indeed in the following we assume this model when working with regular graphs.

As an example, the 
\emph{random regular Ising ferromagnet} is obtained by letting 
$\cX=\{+1,-1\}$ and, for some $\eps\le 1/2$, 
$\psi(x_1,x_2) = 1-\epsilon$ if $x_1=x_2$
and $\psi(x_1,x_2)=\epsilon$ otherwise. 

Specifying a \emph{Poisson ensemble} requires
an alphabet $\cX$, a \emph{density} $\gamma\in\reals_+$, a finite collection of
weights $\{\psi_a(\,\cdot\, ,\,\cdot\,) : \, a\in\cC\}$,
and a probability distribution $\{p(a):\, a\in \cC\}$ over the weights.
In this case $G$ is a multigraph where the number edges among any pair of 
vertices $i$ and $j$ is an independent Poisson random variable of 
parameter $2\gamma/n$. Each loop $(i,i)$ is present with multiplicity 
which is Poisson 
of mean\footnote{Notice that in a typical realization there will be 
only a few loops and non-simple edges.} $\gamma/n$. 
Finally, for each edge in the 
multi-graph, we draw an independent random variable $a$ with distribution
$p(\,\cdot\,)$ and set $\psi_{ij}(\,\cdot\, ,\, \cdot\,)=
\psi_a(\,\cdot\, ,\,\cdot\,)$. 

Two examples of Poisson ensembles to be treated below are 
the \emph{Ising spin glass}, and \emph{antiferromagnetic colorings}
(aka `antiferromagnetic Potts model').
In the first case $\cX = \{+1,-1\}$ and two
type of weights appear with equal probability (i.e. $\cC = \{+,-\}$
and $p(+)=p(-)=1/2$): $\psi_{+}(x_1,x_2) = 1-\eps$ for $x_1=x_2$,
$\psi_{+}(x_1,x_2) =\eps$ for $x_1\neq x_2$, while
 $\psi_{-}(x_1,x_2) = \eps$ for $x_1=x_2$,
$\psi_{-}(x_1,x_2) =1-\eps$ for $x_1\neq x_2$.
For proper colorings $\cX = \{1,\dots,q\}$, and $|\cC|=1$ with
$\psi(x_1,x_2) = 1$
if $x_1\neq x_2$, and $\psi(x_1,x_2) =\epsilon<1$ otherwise (for $\epsilon=0$
one recovers the uniform measure over proper colorings of $G$).

Both graphical model ensembles defined above 
converge locally to trees. In the case of regular models, the corresponding
tree model is an infinite rooted tree of uniform degree $(k+1)$,
each edge being associated the same weight $\psi(\,\cdot\, ,\,\cdot\,)$. 
For Poisson models, the relevant tree is a rooted Galton-Watson
tree with Poisson distributed degrees of mean $2\gamma$. Each edge 
carries the weight $\psi_a(\,\cdot\, ,\,\cdot\,)$ independently with 
probability $p(a)$.

Given such infinite weighted trees, let $\Tree_\ell$, $\ell\ge 0$
be the weighted subgraph obtained by truncating it at depth $\ell$.
One can introduce  random variables
$\uX=\{X_j:j\in\Tree_\ell\}$, by defining $\prob\{\uX=\ux|\Tree_\ell\}$
 as in Eq.~(\ref{eq:GeneralGraphModel}) (with $G$ replaced
by $\Tree_\ell$).
With an abuse of notation we shall call $r$ the 
root of $\Tree_\ell$.
It is natural to ask whether reconstruction on the original 
graphical models and on the corresponding trees are related.
\begin{definition}
Consider a sequence of random graphical models $\{G_N\}$ 
(distributed according either to the regular or to the Poisson ensemble),
and let $\{\Tree_\ell\}$ be the corresponding 
sequence of tree graphical models.
We say that the reconstruction problem is \emph{tree-solvable} for 
the sequence $\{G_N\}$ if there exists $\ve>0$ such that, for any 
$t\ge 0$
\begin{eqnarray}
\Vert \prob_{r,\cBall(r,t)}\{\,\cdot\, ,\,\cdot\,|\Tree_\ell\}-
 \prob_{r}\{\,\cdot\,|\Tree_\ell\} \prob_{\cBall(r,t)}\{\,\cdot\,|\Tree_\ell\}\Vert_{\sTV} >\ve\, ,
\label{eq:TreeReconstruction}
\end{eqnarray}
with positive probability (as $\ell\to\infty$)
\end{definition}
Notice that tree-reconstruction is actually a question on the sequence of
tree graphical models $\{\Tree_{\ell}\}$ indexed by $\ell$. 
The only role of the original random graphs sequence $\{G_N\}$ is to
determine the distribution of $\Tree_{\ell}$.

Despite the similarity of Eqs.~(\ref{eq:TreeReconstruction}) and 
(\ref{eq:GraphReconstruction}), passing from the original graph
to the tree is a huge simplification because $\prob\{\,\cdot\,|\Tree_\ell\}$
has a simple description as mentioned above. For instance,
in the case of a ferromagnetic Ising model, one can sample 
the variables $X_j$ on the tree through a `broadcast' process.
First, generate the root value $X_r$
uniformly at random in $\{+1,-1\}$. Then recursively,
for each node $j$, generate the values of its children 
$\{l\}$ conditional on $X_j=x_j$ by letting $X_l=x_j$ independently with
probability $1-\eps$, and   $X_l=-x_j$ otherwise. Analogous descriptions 
exist for the spin-glass and colorings models.
%
%
\SubSection{Main results}

Our first result is a sufficient condition 
for graph-reconstruction to be equivalent to tree reconstruction.
In order to phrase it, we need to define the `two-replicas type.'
Consider a graphical model $G_N$ and two two iid 
assignments of the variables $\uX^{(1)}$, $\uX^{(2)}$ with common distribution
$\prob\{\,\cdot\,|G_N\}$ 
(we will call them \emph{replicas} following the spin glass
terminology). The \emph{two replica type} is a matrix
$\{\nu(x,y):\, x,y\in\cX\}$ where $\nu(x,y)$ counts the fraction 
of vertices $j$ such that $X^{(1)}_j=x$ and $X^{(2)}_j=y$.
(Conversely, the set of distributions $\nu$ on $\cX\times\cX$ such that 
$N\nu(x,y)\in \naturals$ will be called the set of 
\emph{valid two-replicas types} $\Rep_N$. When we drop the 
constraint $N\nu(x,y)\in\naturals$, we shall use $\Rep$.) 

The matrix $\nu$ is random.
If $\prob\{\,\cdot\,|G_N\}$ were the uniform distribution,
then $\nu$ would concentrate around $\onu(x,y)\equiv 1/|\cX|^2$. 
Our sufficient condition requires this to be approximately true.
\begin{thm}\label{thm:graph<->tree}
Consider a sequence of random Poisson graphical models $\{G_N\}$,
and let $\nu(\,\cdot\, ,\,\cdot\,)$ be the type of  two iid
replicas $\uX^{(1)}$, $\uX^{(2)}$, and $\Delta\nu(x,y)\equiv\nu(x,y)-\onu(x,y)$. 
Assume that, for any $x\in\cX$,
\begin{eqnarray}
\E\left\{\left.\left[\Delta\nu(x,x)-2|\cX|^{-1}\sum_{x'}
\Delta\nu(x,x')\right]^2\right.\right\} 
\stackrel{N}{\to} 0\, . \label{eqn:graph<->tree}
\end{eqnarray}
Then the reconstruction problem for $\{G_N\}$ is solvable if and only if
it is tree-solvable.
\end{thm}
\emph{Remark 1:} Notice that the expectation in Eq.~(\ref{eqn:graph<->tree})
is both over the two replicas $\uX^{(1)}$, $\uX^{(2)}$ 
(which the type $\nu(\,\cdot\,,\,\cdot\,)$ is a function of) conditional
on $G_N$, and over $G_N$. Explicitly $\E\{\cdots\} = 
\E\{\E[\cdots|G_N]\}$.  
\emph{Remark 2:} In fact, as is hinted by the proof, the
condition (\ref{eqn:graph<->tree}) can be weakened, e.g.
$\onu(\,\cdot\,\,\cdot\,)$ can be chosen more generally than the uniform
matrix. This will be treated in a longer publication.

The condition (\ref{eqn:graph<->tree})
emerges naturally in a variety of contexts, a notable one 
being second moment method applied to random constraint 
satisfaction problems \cite{AchlioptasNaorPeres}.
As an example,  consider proper colorings of random graphs. 
In bounding on the colorability threshold, one computes the second
moment of the number of colorings,
and, as an intermediate step, an upper bound on the large deviations of the
type $\nu$.  Oversimplifying, one might interpret 
Theorem \ref{thm:graph<->tree} by saying that, when 
second moment method works, then tree and graph reconstruction are equivalent.
Building on \cite{AchlioptasNaor} we can thus establish the following. 
\begin{thm}\label{thm:Colorings}
Consider antiferromagnetic $q$-colorings of a 
Poisson random graph and let $\gamma_q \equiv(q-1)\log(q-1)$.
Then there exists a set $\Gamma$ of zero (Lebesgue) measure such that 
the following is true. If $\gamma\in [0,\gamma_q)\setminus\Gamma$
and $\eps\in (0,1]$, then
the reconstruction problem is solvable if and only if
it is tree solvable.
\end{thm}
We expect the result to hold down to $\eps=0$ (proper colorings),
with $\Gamma =\emptyset$, but did not prove it because of some 
technical difficulties
(indeed we need a sharper control of $\nu$ that guaranteed 
by \cite{AchlioptasNaor}, and our proof technique, cf. Lemma 
\ref{lemma:Balance}, relied on an average over 
$\gamma$).

The above theorems might suggests that graph 
and tree reconstruction do generally coincide. This expectation
is falsified by the simplest possible example:
the Ising model. This has been studied in depth for trees
\cite{BleherRuizZagrebnov,EvaKenPerSch,Asymmetric}. 
If the tree is regular with degree $(k+1)$, the problem is
solvable if and only if $k(1-2\eps)^2>1$. 
The situation changes dramatically for graphs.
\begin{thm}\label{thm:ferromagnet}
Reconstruction is solvable for random
regular Ising ferromagnets if and only if $k(1-2\eps)>1$.
\end{thm}
This result possibly generalizes to
Ising ferromagnets  on other graphs that converge locally to trees.
The proof of reconstructibility for $k(1-2\eps)>1$ essentially
amounts to finding a bottleneck in Glauber dynamics.
As a consequence it immediately implies that the mixing time is exponential
in this regime. We expect this to be a tight estimate of the threshold
for fast mixing.

On the other hand, for an Ising spin-glass,
the tree and graph thresholds do coincide.
In fact, for an Ising model on a Galton-Watson tree with 
Poisson$(2\gamma)$ offspring distribution, reconstruction is solvable if
and only if $2\gamma(1-2\eps)^2>1$ \cite{EvaKenPerSch}. 
The corresponding graph result is established below. 
\begin{thm}\label{thm:spinglass}
Reconstruction is solvable for Poisson
Ising spin-glasses if $2\gamma(1-2\eps)^2>1$,
and it is unsolvable if $2\gamma(1-2\eps)^2<1$.
\end{thm}
%
%

\vspace{-0.5cm}

\Section{Random graph preliminaries}

Let us start with a few more notations.
Given $i\in V$, and $t\in\naturals$,
$\Ball(i,t)$ is the set of vertices $j$ such that $d(i,j)\le t$
(as well as the subgraph formed by those vertices and by edges that
are not in $\cBall(i,t)$). Further we introduce the set of vertices
$\Edge(i,t) \equiv \Ball(i,t)\cap\cBall(i,t)$.

The proof of Theorem \ref{thm:graph<->tree} relies on two remarkable properties
of Poisson graphical models: the local convergence 
of $\Ball(r,t)$ to the corresponding Galton-Watson tree of depth $t$
(whose straightforward proof we omit), 
and a form of independence of $\cBall(r,t)$ relatively to $\Ball(r,t)$.
Notice that, because of the symmetry of the graph distribution under 
permutation of the vertices, we can fix $r$ to be a deterministic vertex (say, $r=1$).
\begin{propo}\label{prop:tree_convergence}
Let $\Ball(r,t)$ 
be the depth-$t$ neighborhood of
the root in a Poisson random graph $G_N$, and $\Tree_t$
a Galton-Watson tree of depth $t$ and 
offspring distribution {\rm Poisson}$(2\gamma)$.
Given any (labeled) tree $\Tree_*$,
we write $\Ball(r,t) \simeq \Tree_*$ if $\Tree_*$ is obtained 
by the depth-first relabeling of $\Ball(r,t)$ 
following a pre-established 
convention\footnote{For instance one might agree to preserve
the original lexicographic order among siblings}.
Then 
$\prob\{\Ball(r,t)\simeq \Tree_*\}$ converges to
$\prob\{\Tree_t\simeq \Tree_*\}$ as $N\to\infty$.
\end{propo}

\begin{propo}\label{prop:BoundNeighborhood}
Let $\Ball(r,t)$ 
be the depth-$t$ neighborhood of
the root in a Poisson random graph $G_N$. Then, for any $\lambda>0$ 
there exists $C(\lambda,t)$ such that, for any $N$, $M\ge 0$
\begin{eqnarray}
\prob\{|\Ball(r,t)|\ge M\}\le C(\lambda,t)\, \lambda^{-M}\, .
\end{eqnarray}
\end{propo}
\begin{proof}
Imagine to explore $\Ball(r,t)$ in breadth-first fashion.
For each $t$, $|\Ball(r,t+1)|-|\Ball(r,t)|$ is upper bounded by the sum of 
$|\Edge(r,t)|$ iid binomial random variables with parameters
$N-|\Ball(r,t)|$ and $1-e^{-2\gamma/N}\le 2\gamma/N$ 
(the number of neighbors of
each node in $\Edge(r,t)$). Therefore $|\Ball(r,t)|$ is stochastically
dominated by $\sum_{s=0}^tZ_N(s)$, where $\{Z_N(t)\}$
is a Galton-Watson process with offspring distribution 
Binom$(N,2\gamma/N)$. By Markov inequality
$\prob\{|\Ball(r,t)|\ge M\}\le \E\{\lambda^{\sum_{s=0}^t
Z_N(s)}\}\, \lambda^{-M}$.
By elementary branching processes theory 
$ g^N_{t}(\lambda)\equiv\E\{\lambda^{\sum_{s=0}^t Z_N(s)}\}$
satisfies the recursion $g^N_{t+1}(\lambda) = \lambda\xi_N(g^N_t(\lambda))$,
$g^N_0(\lambda) = \lambda$, with $\xi_N(\lambda) = \lambda(1+
2\gamma(\lambda-1)/N)^N$. The thesis follows by 
$g^N_t(\lambda)\le  g_t(\lambda)$, the latter being obtained by replacing
$\xi_N(\lambda)$ with $\xi(\lambda) = e^{2\gamma(\lambda-1)}\ge 
\xi_N(\lambda)$.
\end{proof}
\begin{propo}\label{prop:poisson_independance}
Let $G_N$ be a Poisson random graph on vertex set $[N]$ and 
edge probability $p=2\gamma/N$. Then, conditional on $\Ball(r,t)$,
$\cBall(r,t)$ is a Poisson random graph on vertex set 
$[N]\setminus \Ball(r,t-1)$ and same edge probability.
\end{propo}
\begin{proof}
Condition on $\Ball(r,t) = \Graph(t)$, and let $\Graph(t-1)=\Ball(r,t-1)$
(notice that this is uniquely determined from $\Graph(t)$).
This is equivalent to conditioning on a given edge realization for
any two vertices $k$, $l$ such that $k\in \Graph(t-1)$ and $l\in \Graph(t)$
(or viceversa).

On the other hand, $\cBall(r,t)$ is the graph with vertices set
$[N]\setminus\Graph(t)$ and edge set $(k,l)\in G_N$ such that
$k,l\not\in \Graph(t-1)$. Since this set of vertices couples is disjoint
from the one we are conditioning upon, and by independence of
edges in $G_N$,  the claim follows.
\end{proof}
%

%

\vspace{-0.4cm}

\Section{Proof of Theorem \ref{thm:graph<->tree}}

We start from a simple technical result.
\begin{lemma}\label{lemma:Simple}
Let $p$, $q$ be probability distribution over a finite set $\cS$,
and denote by $q_0(x) = 1/|\cS|$ the uniform distribution over the same set.
Define $\hp(x) \equiv p(x)q(x)/z$, where $z\equiv \sum_x p(x)q(x)$.
Then 
$||\hp-p||_{\sTV} \le 3|\cS|^2\, ||q-q_0||_{\sTV}$.
\end{lemma}
\begin{proof} 
Since $||\hp-p||_{\sTV}\le 1$ we can assume, without loss of generality,
that $ ||q-q_0||_{\sTV}\le (2|\cS|)^{-1}$. If we write $p(x) = p(x)q_0(x)/z_0$,
with $z_0=1/|\cS|$, then $|z-z_0|\le |\sum_x[p(x)q(x)-p(x)q_0(x)]|
\le ||q-q_0||_{\sTV}$ and in particular $z\ge z_0/2$. From triangular 
inequality we have on the other hand
\begin{align*}
||\hp-p||_{\sTV} \le \frac{1}{2}\left|z^{-1}-z_0^{-1}\right|
+\frac{1}{2z_0}\sum_x p(x)|q(x)-q_0(x)|\, .
\end{align*}
Using $|x^{-1}-y^{-1}|\le |x-y|/\min(x,y)^2$, the first term is bounded by 
$2|z-z_0|/z_0^2\le 2|\cS|^2||q-q_0||_{\sTV}$. The second is at most
$|\cS|\, ||q-q_0||_{\sTV}$ which proves the thesis.
\end{proof}

In order to prove Theorem \ref{thm:graph<->tree} we 
first establish that, under the condition (\ref{eqn:graph<->tree}),
any (fixed) subset of the variables $\{X_1,\dots,X_N\}$
is (approximately) uniformly distributed.
\begin{propo}\label{degenerescence}
Let $i(1),\dots,i(k)\subseteq [N]$ be any (fixed)
set of vertices, and $\xi_1,\dots, \xi_k\in\cX$.  
Then, under the hypotheses of Theorem 
\ref{thm:graph<->tree}, for any $\ve>0$
\begin{eqnarray}
\left|\prob_{i(1),\dots,i(k)}\{\xi_1,\dots, \xi_k|G_N\} - 
\frac{1}{|\cX|^k} \right|\le \ve\, ,\label{eqn:degenerescence}
\end{eqnarray}
with high probability.
\end{propo}
\begin{proof}
Given two replicas $\uX^{(1)}$, $\uX^{(2)}$, define, for $\xi\in\cX$
(with $\ind_{\cdots}$ the indicator function)
\begin{eqnarray*}
\sQ(\xi) = \frac{1}{N}\sum_{i=1}^N  \left\{
\ind_{X^{(1)}_i=\xi}-\frac{1}{|\cX|}\right\}\left\{
\ind_{X_i^{(2)}=\xi}-\frac{1}{|\cX|}\right\}\, .\label{eq:Overlap}
\end{eqnarray*}
Notice that $\sQ(\xi) = \Delta\nu(\xi,\xi)-(2/|\cX|)
\sum_{x}\Delta\nu(\xi,x)$ is the quantity in Eq.~(\ref{eqn:graph<->tree}).
Therefore, under the hypothesis of Theorem \ref{thm:graph<->tree},
$\E\{\sQ(\xi)^2\}\stackrel{N}{\to}0$. Further, since
$|Q(\xi)|\le 1$ and using Cauchy-Schwarz, for any
$\xi_1,\dots,\xi_k\in\cX$
\begin{eqnarray*}
\left|\E\left\{\sQ(\xi_1)\cdots\sQ(\xi_k)\right\}\right| \le
\E|\sQ(\xi_1)|\stackrel{N}{\to} 0\, .\label{eq:kOverlaps}
\end{eqnarray*}

If we denote by $\sQ_i(\xi)$ the quantity on the right hand side of the sum
in Eq.~(\ref{eq:Overlap}) then $\sQ(\xi)$
is the uniform average of $\sQ_i(\xi)$ over a uniformly random $i\in [N]$.
By symmetry of the graph distribution with respect 
to permutation of the vertices in $[N]$, and since $|\sQ(\xi)|\le 1$ 
we get 
\begin{eqnarray*}
\E\left\{\sQ(\xi_1)\cdots\sQ(\xi_k)\right\}= 
\E\left\{\sQ_{i(1)}(\xi_1)\cdots \sQ_{i(k)}(\xi_k)\right\}+\ve_{k,N}\\
=\E\{\E\{ \prod_{a=1}^k (\ind_{X_{i(a)}=\xi_a}-|\cX|^{-1}) |G_N\}^2\}+\ve_{k,N}\, ,
\end{eqnarray*}
where $|\ve_{k,N}|$ is upper bounded by the probability that $k$
random variable uniform in $[N]$ are not distinct (which is $O(1/N)$).
Therefore the expectation on right hand side vanishes as $N\to\infty$
as well, which implies (since the quantity below is, again, bounded by $1$)
\begin{eqnarray}
\left|\E\left\{\left. 
\prod_{a=1}^k (\ind_{X_i(a)=\xi_a}-|\cX|^{-1}) \right|G_N\right\}
\right|\le \ve\label{eq:BoundCorrFun}
\end{eqnarray}
with high probability for any $\ve>0$. The proof is completed by
noting that the left hand side of Eq.~(\ref{eqn:degenerescence}) 
can be written as
\begin{eqnarray*}
\left|\sum_{\emptyset\neq U\subseteq [k]}
\E\left\{\left. 
\prod_{a\in U} (\ind_{X_{i(a)}=\xi_{i(a)}}-|\cX|^{-1})\right|G_N \right\}
\right|\le 2^k\ve\, ,
\end{eqnarray*}
where the last bound holds whp thanks to Eq.~(\ref{eq:BoundCorrFun}) and 
$\ve$ can eventually be rescaled.
\end{proof}

In order to write the proof Theorem \ref{thm:graph<->tree} we need to 
introduce a few shorthands. Given a graphical model $G_N$, and $U\subseteq[N]$,
we let
$\mu_U(\ux_U) \equiv \prob\left\{\uX_U=\ux_U|G_N\right\}$
(omitting subscripts if $U=V$). If $r$
is its root, $\ell\in\naturals$ and $U\subseteq\Ball(r,\ell)$, we define
$\mu^{<}_U(\ux_U) \equiv \prob\left\{\uX_U=\ux_U|\Ball(r,\ell)\right\}$
(i.e. $\mu^{<}$ is the distribution obtained by restricting the product
in (\ref{eq:GeneralGraphModel}) to edges $(i,j)\in \Ball(r,\ell)$).
Analogously 
$\mu^{>}_U(\ux_U) \equiv \prob\left\{\uX_U=\ux_U|\cBall(r,\ell)\right\}$.
Finally for $U\subseteq [N]$, we let $\rho_U(\ux_U)=1/|\cX|^{|U|}$
be the uniform distribution on $\cX^U$.
\begin{lemma}
Let $G_N$ be a graphical model rooted at $r$, and
$\ell\in\naturals$.
Then  for any $t\le \ell$,
\begin{align}
\Big|||\mu_{r,\cBall(r,t)}-&\mu_r\mu_{\cBall(r,t)}||_{\sTV}-
||\mu^<_{r,\cBall(r,t)}-\mu^<_r\mu^<_{\cBall(r,t)}||_{\sTV}\Big| \le\nonumber\\
&\le 9|\cX|^{2|\Ball(r,\ell)|}\, ||\mu^{>}_{\Edge(r,\ell)}-\rho_{\Edge(r,\ell)}||_{\sTV}\, .\label{eq:treeBoundary}
\end{align}
\end{lemma}
\begin{proof}
First notice that, by elementary properties  of the total variation distance,
$||\mu_{U}-\mu^<_{U}||_{\sTV}\le
||\mu_{\cBall(r,\ell)}-\mu^<_{\Ball(r,\ell)}||_{\sTV}$
for any $U\subseteq \Ball(r,\ell)$. Applying this remark and triangular 
inequality, the left hand side of 
Eq.~(\ref{eq:treeBoundary}) can be upper bounded by 
$3\, ||\mu_{\cBall(r,\ell)}-\mu^<_{\Ball(r,\ell)}||_{\sTV}$.
Next notice that, as a consequence of Eq.~(\ref{eq:GeneralGraphModel}) 
and of the fact that $\Ball(r,\ell)$ and $\cBall(r,\ell)$ are edge disjoint
(and using the shorthands $\Ball(\ell)$ and $\Edge(\ell)$ for
 $\Ball(r,\ell)$ and $\Edge(r,\ell)$)
\begin{eqnarray*}
 \mu_{\Ball(\ell)}(\ux_{\Ball(\ell)}) 
= \frac{\mu^{<}_{\Ball(\ell)}(\ux_{\Ball(\ell)}) \mu^{>}_{\Edge(\ell)}(\ux_{\Edge(\ell)}) }
{\sum_{\ux'_{\Ball(\ell)}}\mu^{<}_{\Ball(\ell)}(\ux'_{\Ball(\ell)}) \mu^{>}_{\Edge(\ell)}(\ux'_{\Edge(\ell)})}\, .
\end{eqnarray*}
We can therefore apply
Lemma \ref{lemma:Simple} whereby $p$ is $\mu^{<}_{\Ball(\ell)}$, 
$\hp$ is $\mu_{\Ball(\ell)}$, $q$ is
$\mu^{>}_{\Edge(\ell)}$, and $\cS=\cX^{\Ball(\ell)}$. This
proves the thesis.
\end{proof}

\vspace{-0.2cm}

\begin{proof}[Proof of Theorem \ref{thm:graph<->tree}] 
Let $\Delta_N$ denote  the left hand side of  
Eq.~(\ref{eq:treeBoundary}). We claim that its 
expectation (with respect to a random graph $G_N$) vanishes as $N\to\infty$. 
Since the probability that $\Ball(r,\ell)\ge M$ can be made arbitrarily 
small by letting $M$ large enough, cf. Lemma \ref{prop:BoundNeighborhood}, 
and using the fact that the left hand side of Eq.~(\ref{eq:treeBoundary})
is bounded by $1$, it is sufficient to prove that
\begin{eqnarray*}
\sum_{|\Graph|\le M} \prob\{\Ball(r,\ell)=\Graph\}
\E\{\Delta_N|\Ball(r,\ell)=\Graph\}\le\\
\le  K^{M+1}
\sum_{|\Graph|\le M} 
\E\{ ||\mu^{>}_{\Edge(r,\ell)}-\rho_{\Edge(r,\ell)}||_{\sTV}
|\Ball(r,\ell)=\Graph\}\, ,
\end{eqnarray*}
vanishes as $N\to\infty$. Each term in the sum is the expectation, with 
respect to a random graph over $N-|\Graph|\ge N-M$ vertices
of the total variation distance between
the joint distribution of a fixed set of vertices, and the uniform 
distribution (for $\Edge = \Edge(r,\ell)$):
\begin{eqnarray*}
||\mu^{>}_{\Edge}-\rho_{\Edge}||_{\sTV}=
\frac{1}{2}\sum_{\ux_{\Edge}}\left|\prob_{\Edge}
\{\ux_{\Edge}|\cBall(i,\ell)\}
-|\cX|^{-|\Edge|}\right|\, .
\end{eqnarray*}
This vanishes by Lemma \ref{degenerescence}, thus proving the above claim.

This implies that 
there exists $\ve>0$ such that 
$||\mu_{r,\cBall(r,t)}-\mu_r\mu_{\cBall(r,t)}||_{\sTV}\ge
\ve$ with positive probability, if and only if
there exists $\ve'>0$ such that 
$||\mu^<_{r,\cBall(r,t)}-\mu^<_r\mu^<_{\cBall(r,t)}||_{\sTV}\ge \ve'$
with positive probability.
In other words, since 
$\mu(\,\cdot\,) \equiv \prob\{\,\cdots\,|G_N\}$, reconstruction is 
solvable if and only if 
$||\mu^<_{r,\cBall(r,t)}-\mu^<_r\mu^<_{\cBall(r,t)}||_{\sTV}\ge \ve'$
with positive probability.

Finally, recall that 
$\mu^{<}(\,\cdot\,)\equiv\prob\{\,\cdot\,|\Ball(r,\ell)\}$ and that 
$\Ball(r,\ell)$ converges in distribution to $\Tree(\ell)$, by 
Lemma \ref{prop:tree_convergence}. Since 
$||\mu^<_{r,\cBall(i,t)}-\mu^<_r\mu^<_{\cBall(r,t)}||_{\sTV}$
is a bounded function of $\cBall(r,t)$ (and using as above Lemma 
\ref{prop:BoundNeighborhood} to reduce to a finite set of graphs),
it converges in distribution to
$\Vert \prob_{r,\cBall(r,t)}\{\,\cdot\, ,\,\cdot\,|\Tree_\ell\}-
 \prob_{r}\{\,\cdot\,|\Tree_\ell\} \prob_{\cBall(r,t)}\{\,\cdot\,|\Tree_\ell\}\Vert_{\sTV}$. We conclude that $||\mu^<_{r,\cBall(i,t)}-\mu^<_r
\mu^<_{\cBall(r,t)}||_{\sTV}\ge \ve'$ with positive probability
if and only if reconstruction is tree solvable, thus proving the thesis.
\end{proof}
\vspace{-0.6cm}

%
%
\Section{Two successful applications}

\vspace{-0.1cm}

\SubSection{The Ising spin glass}

\vspace{-0.1cm}

\begin{proof}[Proof of Theorem \ref{thm:spinglass}]The behavior of 
spin-glasses
on  Poisson random graphs has been studied extensively in \cite{Guerra}. In 
particular, the \emph{two-replica overlap} $q_{12} = 
N^{-1}\sum_i X_i^{(1)} X_i^{(2)}$ satisfies  $ \E\{q_{12}^2\} = 
O\left(N^{-1}\right)$ for $2\gamma(1-2\eps)^2 < 1 $
(``high-temperature'' region).
It is easy to check that 
the quantity in the expectation  in Eq.~(\ref{eqn:graph<->tree}) 
equals
$(q_{12}/4)^2$ both for $x=+1$, and $-1$. Hence, if $2\gamma(1-2\eps)^2 < 1$, 
Theorem \ref{thm:graph<->tree} 
applies. Since tree reconstruction is unsolvable in that case
\cite{EvaKenPerSch} (notice that on trees, reconstruction for 
the spin glass model and the ferromagnet are equivalent), we 
obtain that graph reconstruction is unsolvable as well.

Conversely, suppose that $2\gamma(1-2\eps)^2 > 1$.
First assume $ \E\{q_{12}^2\}\stackrel{N}{\to}  0$. 
Since tree reconstruction is solvable in this case,
Theorem \ref{thm:graph<->tree} would  imply that graph 
reconstruction is solvable as well. It is thus sufficient to prove that graph 
reconstruction 
is solvable if $ \E\{q_{12}^2\} \not\to 0$. 
Equivalently,
if for any $\ve>0$, there exists $t=t(\ve)$
such that $||\prob_{r,\cBall(r,t)}\{\,\cdot\, ,\,\cdot\,|G_N\}-
\prob_{r}\{\,\cdot\,|G_N\}\prob_{\cBall(r,t)}\{\,\cdot\,|G_N\}||_{\sTV}\le\ve$
with high probability, then
$ \E\{q_{12}^2\}\stackrel{N}{\to}  0$.

Since $q_{12}$ is the average of $X_{i}^{(1)}X_{i}^{(2)}$
over a uniformly random $i\in[N]$, then $\E\{ q_{12}^2\}$
is the average of  $\E\{X_{r}^{(1)}X_{r}^{(2)}X_{j}^{(1)}X_{j}^{(2)}\}=
\E\{\E\{X_{r}X_{j}|G_N\}^2\}$ over $r$, $j\in [N]$ uniform and independent.
Fixing $t=t(\ve)$ as above, we can neglect the 
probability that $j\in\Ball(r,t)$, since this is  $N^{-1}$
times the expected size of $\Ball(r,t)$, that is bounded by Lemma 
\ref{prop:BoundNeighborhood}. Therefore
$\E\{q_{12}^2\} = \E\left\{\E[X_rX_j|G_N]^2\, \ind_{j \notin 
\Ball(r,t)}\right]+o(1)$. On the other hand, if
$j\not\in\Ball(r,t)$
\begin{align*}
&\big|\E[X_rX_j|G_N]\big|
 \le\\
&\le ||\prob_{r,\cBall(r,t)}\{\,\cdot\, ,\,\cdot\,|G_N\}-
\prob_{r}\{\,\cdot\,|G_N\}\prob_{\cBall(r,t)}\{\,\cdot\,|G_N\}||_{\sTV}\, .
\end{align*}
We deduce that $\big|\E[X_rX_j|G_N]\big|\le \ve$ with high 
probability ant hence $\lim_N\E\{q_{12}\}\le \ve^2$
The thesis follows by recalling that $\ve$ is an arbitrary positive number.
\end{proof}
\vspace{-0.25cm}

%
%
\SubSection{$q$-colorings of Poisson random graphs}

The application Theorem \ref{thm:graph<->tree} to this case require some 
technical groundwork. For space reasons we limit to quoting the 
results deferring the proofs to a complete publication. 
We denote by $U(\ux)$ be the number of monochromatic edges under 
coloring $\ux$, by $Z= \sum_{\ux}\eps^{U(\ux)}$ the partition 
function, by $Z_{\rm b}$ the modified partition function where the sum is 
restricted to balanced
colorings (such that each color occupies $N/q$ vertices), and
by $Z_2(\nu) = \sum_{\ux^{(1)},\ux^{(2)}}^{\nu} \eps^{U(\ux^{(1)})+U(\ux^{(2)})}$,
where the sum is restricted to couples of colorings with two-replica type 
$\nu=\{\nu(x,y)\}_{x,y\in [q]}$.
As above we denote by $\onu(x,y) =1/q^2$ for any $x,y\in[q]$ the uniform 
matrix. 
Finally we introduce the following function over two-replica types
$\nu$ (i.e. over $q\times q$ matrices with non-negative entries 
normalized to one):
\begin{align*}
\phi(\nu) = -\sum_{xy}&\nu(x,y)\log\nu(x,y)+\\
&+\bar{\gamma}\log\Big\{
1-\bar{\eps}F(\nu)+\bar{\eps}^2
\sum_{x,y}\nu(x,y)^2\Big\}\, ,
\end{align*}
where $\bar{\eps} = 1-\eps$ and 
$F(\nu) \equiv2\sum_x(\sum_y\nu(x,y))^2$.

The first two preliminary remarks are
 a combinatorial calculation that straightforwardly generalizes
the result of \cite{AchlioptasNaor} for proper colorings,
and  a good estimate on the balanced partition function.
\begin{lemma}\label{lemma:2ndColoring}
Let $M$ be the number of edges in a Poisson graph. Then
$\E[ Z_2(\nu)|M=\bar{\gamma}N ]\le K\, e^{n\phi(\nu)}$.
\end{lemma}
\begin{lemma}\label{lemma:ColoringPartFun}
Let $\gamma<\gamma_q=(q-1)\log(q-1)$. Then,
for any $\xi>0$, $Z_{\rm b}\ge e^{N[\phi(\onu)-\xi]/2}$
with high probability. Further, 
$\E\,\log Z_{\rm b}\ge N[\log q+\gamma\log(1-\bar{\eps}/q)] + O(N^{1/2})$.
\end{lemma}
Our last remark is that, for $\gamma<\gamma_q$, balanced
colorings dominate the measure $\mu$.
\begin{lemma}\label{lemma:Balance}
Let $\Gamma$ and $\gamma_q$ be as in the statement of Theorem
\ref{thm:Colorings}, and $\nu(x,y)$ be as in \ref{thm:graph<->tree}.
Then, for any $\gamma\in [0,\gamma_{q})\setminus \Gamma$,
any $x\in \cX$, and any $\delta >0$, $|\sum_y\nu(x,y)-q^{-1}|\le \delta$
with high probability.
\end{lemma}
\begin{proof}
Recall 
that, if $X$ is a Poisson random variable of mean $\lambda$, 
and $f(\lambda) \equiv \E\, F(X)$, then $f'(\lambda) = \E[F(X+1)-F(X)]$.
Further notice that, if $Z(G)$ is the partition function for the
graph $G$, then 
\begin{eqnarray}
Z(G\cup (ij)) = Z(G)\, \big[1-\bar{\eps}\prob\{X_i=X_j|G\}\,\big]\, .
\end{eqnarray}
Applying these identities to $\alpha(\gamma) \equiv N^{-1}\E\log Z(G_N)$,
we get
\begin{eqnarray}
\frac{\de\alpha(\gamma)}{\de\gamma} = 
\frac{1}{N^2}\sum_{i,j}\E\log\left\{1-\bar{\eps}\prob\{X_i=X_j|G_N\}\right\}\, .
\end{eqnarray}
Since $\alpha(0) = \log q$, by using Jensen inequality we get
\begin{eqnarray}
\alpha(\gamma)\le \log q+\int_0^{\gamma}\log\{1-\bar{\eps}\, \E\,[A(\gamma')]
\}\de\gamma'\, ,\label{eq:SumRule}
\end{eqnarray}
where (the dependence on $\gamma$ being through the 
distribution of $G_N$ and hence of $\uX$)
\begin{eqnarray*}
A(\gamma)\equiv \frac{1}{N^2}\sum_{ij}\ind\{X_i=X_j\} = 
\sum_x\big(\sum_y \nu(x,y)\big)^2\, .
\end{eqnarray*}
On the other hand, for $\gamma<\gamma_q$ 
\begin{eqnarray}
\alpha(\gamma) \ge \log q+\gamma\log(1-\bar{\eps}/q) + O(N^{-1/2})\, .
\label{eq:RealPsi}
\end{eqnarray}
This follows from $Z\ge Z_{\rm b}$ together with Lemma
\ref{lemma:ColoringPartFun}.

From Eq.~(\ref{eq:SumRule}) and  (\ref{eq:RealPsi}), and since 
$A(\gamma)\ge 1/q$ by definition, we get $A(\gamma)\le (1+\ve)/q $
with high probability for any $\gamma\in [0,\gamma_{q})\setminus \Gamma$, where 
$\Gamma$ has zero Lebesgue measure. Finally by Cauchy-Schwarz
\begin{eqnarray*}
\sum_x\Big|\sum_y\nu(x,y)-\frac{1}{q}\Big|
\le q\, \sum_x\Big[\sum_y\nu(x,y)-\frac{1}{q}\Big]^2\le\\
\le q\, A(\gamma)-1\le \ve\, ,
\end{eqnarray*}
where the last inequality holds with high probability by the above.
\end{proof}

\begin{lemma}\label{lemma:UBProbCol}
Let $\cA\subset\Rep$ be any subset of the set of two replicas types,
and assume $\sup_{\nu\in \cA}\phi(\nu)<\phi(\onu)$. Then 
$\nu\not\in\cA$ with high probability.
\end{lemma}
\begin{proof}
Fix $\xi>0$ and denote by $\Typ$ the set of graphs such that 
 $Z\ge Z_{\rm b}>e^{N[\phi(\onu)-\xi]/2}$, and that the number of edges $M$
satisfies $|M-N\gamma|\le o(N)$. 
For any $G_N\in\Typ$, we have (denoting, with an abuse of notation,
the two replica type of $\ux^{(1)}$ and $\ux^{(2)}$ by $\nu$ as well)
\begin{align*}
\prob&\{\nu\in \cA| G_N\}
=\frac{1}{Z^2}\sum_{\ux^{(1)},\ux^{(2)}} \eps^{ U(\ux^{(1)})+U(\ux^{(2)})}
\,\ind(\nu\in\cA)\,\\
&\le e^{-N[\phi(\onu)-\xi]}\sum_{\ux^{(1)},\ux^{(2)}} 
\eps^{ U(\ux^{(1)})+U(\ux^{(2)})}\,
\ind(\nu\in\cA)\, .
\end{align*}
On the other hand, it follows from Lemma  \ref{lemma:ColoringPartFun} that 
$\prob\{\nu\in\cA\}\le \E\{\prob\{\nu\in\cA|G_N\}\,\ind_{G_N\in\Typ}\}+o_N(1)$.
Using Lemma \ref{lemma:2ndColoring}, this implies
\begin{align*}
\prob&\{\nu\in \cA\} \le Ke^{N\xi}\!\!\!\sum_{\nu\in\Rep_N\cap\cA}\!\!\!\!
e^{-N[\phi(\onu)-\phi(\nu)+o_N(1)]}+o_N(1)\, ,
\end{align*}
where the $o(1)$ term in the exponent accounts for the fact that 
$M = N[\gamma+o(1)]$ (as $\phi(\nu)$ is continuous in $\ogamma$). 
The thesis follows by choosing $\xi = \inf_{\nu\in\cA}[\phi(\onu)-\phi(\nu)]
/2>0$ and noting that the number of terms in the sum is at most
$|\Rep_N|=O(N^q)$.
\end{proof}

\begin{proof}[Proof of Theorem \ref{thm:Colorings}]
The quantity appearing in the expectation
in Eq.~(\ref{eqn:graph<->tree}) is upper  bounded by 
$3\max_{x,y}|\nu(x,y)-\onu(x,y)|$.
We will prove that, for $\gamma\in[0,\gamma_q)\setminus\Gamma$,
for any $\delta>0$, and any $x,y\in\cX$,
$|\nu(x,y)-\onu(x,y)|\le \delta$ with high probability, which implies 
the sufficient condition in Theorem \ref{thm:graph<->tree}. 

Notice that $F(\nu)\ge 2/q$ with $F(\nu) = 2/q$ if and only if 
$\sum_y\nu(x,y)=1/q$.
Because of Lemma  \ref{lemma:Balance},
$F(\nu)\le q^{-1}+\delta'$ with high probability 
for any $\delta'>0$ (to be fixed below). The thesis follows by 
applying Lemma \ref{lemma:UBProbCol} to the event 
$\cA = \{ |\nu(x,y)-\onu(x,y)|> \delta;\, F(\nu)\le q^{-1}+\delta' \}$,
thus showing that $\nu\not\in\cA$ with high probability and
hence $|\nu(x,y)-\onu(x,y)|\le \delta$.

We are left with the task of checking the hypothesis of
Lemma \ref{lemma:UBProbCol}, namely  $\sup_{\nu\in \cA}\phi(\nu)<\phi(\onu)$.
Achlioptas and Naor proved that,
if $F(\nu) = q^{-1}$ (i.e. the column and row sums of $\nu$ are all equal),
and $\gamma<\gamma_q$, then $\phi(\nu)\le \phi(\onu)-A'||\nu-\onu||_{\sTV}^2$
for some $A'>0$ (see \cite{AchlioptasNaor}, Theorem 7 and discussion below).
Under the condition $||\nu-\onu||_{\sTV}\ge |\nu(x,y)-\onu(x,y)|\ge \delta$
$\phi(\nu)< \phi(\onu)$, always subject to $F(\nu) = q^{-1}$.
But by continuity of $F(\nu)$ and $\psi(\nu)$, we can chose
$\delta'$ small enough such that $\phi(\nu)< \phi(\onu)$
for $F(\nu)\le q^{-1}+\delta'$ as well.
\end{proof}
%
%
\vspace{-0.25cm}

\Section{The case of the Ising ferromagnet}
\label{sec:IsingFerro}

We now set out to demonstrate a counter-example to the graph-tree 
reconstruction equivalence encountered above: 
the reconstruction threshold for the random 
$(k+1)$-regular Ising ferromagnet is $k(1-2\eps)=1$
(Theorem \ref{thm:ferromagnet}). 
It is convenient to use a symmetric notation by letting 
$\theta\equiv 1-2\eps>0$, and to generalize the model introducing 
a second parameter $\lambda\ge 0$ (corresponding to a `magnetic
field' in the physics terminology). We then let 
$\psi(+,+) = (1+\lambda)(1+\theta)$, $\psi(-,-) = (1-\lambda)(1+\theta)$,
and  $\psi(+,-) = \psi(-,+) = (1-\theta)$. The original problem
is recovered by letting $\lambda=0$. In terms of these parameters the 
distribution of $\uX$ reads
\begin{align}
\prob\{\uX=\ux|G_N\} = \frac{1}{Z_{\theta,\lambda}}
\theta_+^{e_=(\ux)}
\theta_-^{e_{\neq}(\ux)}\lambda_+^{n_+(\ux)}\lambda_-^{n_-(\ux)}\, ,
\label{eq:IsingExplicit}
\end{align}
whereby $\theta_\pm \equiv 1\pm \theta$, $\lambda_\pm\equiv 1\pm\lambda$,
$n_\pm(\ux)$ is the number of vertices with $x_i=\pm$, and
$e_=(\ux)$ (respectively $e_{\neq}(\ux)$) denotes the number of edges 
$(ij)$ with $x_i=x_j$ (respectively $x_i\neq x_j$).

A crucial role is played by of the partition function $Z_{\theta,\lambda}$
(defined by the normalization condition of $\prob\{\,\cdot\,|G_N\}$)
as well as the constrained partition functions
\begin{align}
\Zh_{\theta,M}& =\!\!\! \sum_{n_+(\ux)-n_-(\ux)=NM}\!\!\!
\theta_+^{e_=(\ux)}\theta_-^{e_{\neq}(\ux)}\, .\label{eq:MagnPartFun}
\end{align}
The rationale for introducing  $\Zh_{\theta,M}$ is that it allows to 
estimate the distribution of the number of $+$'s (or $-$'s)
through the identity (valid for $\lambda=0$)
$\prob\{n_+(\uX)-n_-(\uX) = NM|G_N\} = \Zh_{\theta,M}/Z_{\theta,0}$.

The first technical tool is a well known tree calculation.
\begin{lemma}\label{lmm:recursion}
Assume $\Tree$ to be a regular tree with branching $k$
and depth $t$, rooted at $r$, let $L$ be its leaves,
and let $\prob\{\uX=\ux|\Tree\}$ be defined as in Eq.~(\ref{eq:IsingExplicit})
whereby $n_{\pm}(\ux)$ does not count variables in $\ux_L$. 
For $h_0\in[-1, +1]$,
let $F_{h_0}(\ux_L)$ be the law of $|L|$ iid Bernoulli variables of parameter
$(1+h_0)/2$, and
define $\prob_{h_0}\{\uX=\ux|\Tree\}\equiv 
\prob\{\uX=\ux|\Tree\}\, F_{h_0}(\ux_L)/C$ (with $C$ a normalization constant).

Then $\prob_{h_0}\{X_i=\pm|\Tree\} = (1\pm h_t)/2$,
where $h_t\equiv f_{\theta,\lambda}^{\circ t}(h_0)$ 
($f_{\theta,\lambda}^{\circ t}$ being the $t$-fold composition of 
$f_{\theta,\lambda}$) and
\begin{eqnarray}
f_{\theta,\lambda}(h) \equiv \frac{(1+\lambda)(1+\theta h)^k-(1-\lambda)(1-\theta 
h)^k}{(1+\lambda)(1+\theta h)^k+(1-\lambda)(1-\theta 
h)^k}\, .\label{eq:fdef}
\end{eqnarray}
\end{lemma}
\begin{lemma}\label{lemma:fproperties}
For any $\lambda,\ve>0$, 
$|f_{\theta,\lambda}^{\circ t}(+1)-f_{\theta,\lambda}^{\circ t}(0)|\le \ve$
for $t$ large enough. Further, for $k\theta\le 1$ 
and any $h\in[-1,+1]$ $|f_{\theta,0}^{\circ t}(h)|\le \ve$ for $t$ large enough.
\end{lemma}
The following estimate of $\Zh_{\theta,M}$ is a standard exercise in combinatorics, whose proof we omit.
\begin{lemma}\label{lmm:Zmagnetized}
There exist $C,D>0$ independent of $N$ such that 
$\E \, \Zh_{\theta,M} \le C\, 2^{N}\, e^{DN\,M^2}$.
\end{lemma}
\begin{lemma}\label{lemma:Energy}
For any $\lambda\ge 0$, let   $h_*$ the unique non-negative
solution of $h_*=f_{\theta,\lambda}(h_*)$ and
define $e_{\theta,\lambda}(h) = (\theta+h^2)/(1+\theta h^2)$. 
Then, for any $\ve>0$, and 
a uniformly random edge $(i,j)\in E$,
$|\E\{X_i X_j|G_N\}-e_{\theta,\lambda}(h_*)|\le \ve$ whp.
\end{lemma}
\begin{proof}
Let $h_{t,+} = f_{\theta,\lambda}^{\circ t}(+1)$ and 
$h_{t,0} = f_{\theta,\lambda}^{\circ t}(0)$.
Since $e_{\theta,\lambda}(h)$ is continuous
in $h$, and because of Lemma \ref{lemma:fproperties}, we can fix
$t=t(\ve)$ in such a way that $|e_{\theta,\lambda}(h_{t,+})-e_{\theta,\lambda}(h_{t,0})|\le \ve$. We will show  that, whp,
$e_{\theta,\lambda}(h_{t,0})\le
\E\{X_i X_j|G_N\}\le e_{\theta,\lambda}(h_{t,+})$, thus proving the
thesis, since (by monotonicity of $f_{\theta,\lambda}(\,\cdot\,)$ and
 $e_{\theta,\lambda}(\,\cdot\,)$) 
$e_{\theta,\lambda}(h_{t,0})\le
e_{\theta,\lambda}(h_*)\le e_{\theta,\lambda}(h_{t,+})$ as well.

In order to prove our claim, notice that $\Ball = \Ball(i,t)\cup  \Ball(j,t)$
is whp a tree (obtained by joining through their roots two regular trees 
with branching $k$ and depth $t$), and denote by $\Edge$ its leaves. 
Griffiths inequalities imply \cite{Georgii} that $\E\{X_iX_j|G_N\}$
can be lower bounded  by replacing $G_N$ with any subgraph,  
and upper bounded conditioning on $X_k=+1$ for any set of vertices $k$.
In particular we have
\begin{eqnarray*}
\E\{X_iX_j|\Ball\}\le
\E\{X_iX_j|G_N\}\le \E\{X_iX_j|\uX_{\Edge}=\underline{+1},\Ball\}\, ,
\end{eqnarray*} 
where (in the upper bound) we emphasized that, by the Markov property of
$\prob\{\,\cdot\,|G_N\}$, $\uX_{\Ball}$ is conditionally independent of
$G_N\setminus \Ball$, given $\uX_{\Edge}$.

The proof is finished by evaluating the upper and lower bound
under the assumption, mentioned above, that $\Ball$ is a tree.
This can be done through a dynamic programming-type calculation,
which we omit from this abstract. The final result is
$\E\{X_iX_j|\Ball\} = e_{\theta,\lambda}(f^{\circ t}_{\theta,\lambda}(\lambda))
\ge e_{\theta,\lambda}(h_{t,0})$ and 
$\E\{X_iX_j|\uX_{\Edge}=\underline{+1},\Ball\} 
= e_{\theta,\lambda}(h_{t,+})$, which finishes the proof.
\end{proof}
\begin{lemma}\label{lmm:Ztotal}
For any $\lambda\ge 0$, let   $h_*$ be the unique non-negative
solution of $h_*=f_{k,\lambda}(h_*)$ and $\varphi(\theta,\lambda)\equiv
\phi(\theta,\lambda,h_*)$, where
\begin{align*}
\phi(\theta,&\lambda,h)  \equiv 
-\frac{k+1}{2}\log(1+\theta h^2)+\\
&+\log[(1+\lambda)(1+\theta h)^{k+1}+ (1-\lambda)(1-\theta h)^{k+1}]\, .
\end{align*}
Then $(Z_{\theta,\lambda}/e^{N\varphi})\in[e^{-N\ve}, e^{N\ve}]$
whp for any $\ve>0$.
\end{lemma}
\begin{proof}
Let $\varphi_N(\theta,\lambda)\equiv N^{-1}\log Z_{\theta,\lambda}$.
The proof consists in showing that 
$\E|\varphi_N(\theta,\lambda)
-\varphi(\theta,\lambda)|\stackrel{N}{\to} 0$, whence the thesis follows 
by Markov inequality applied to the event 
$(Z_{\theta,\lambda}/e^{N\varphi})\notin[e^{-N\ve}, e^{N\ve}]$.

It can be proved that $\varphi_N(\theta,\lambda)$ is uniformly
(in $N$) 
continuous with respect to $\lambda$. We can therefore restrict, without
loss of generality to $\lambda>0$.

Next we notice that the above claim is true for 
$\theta=0$ by elementary algebra: 
$Z_{0,\lambda} = 2^{N} =e^{N\varphi(0,\lambda)}$.  
In  order to prove it for $\theta>0$, we write
(omitting the dependence on $\lambda$ that is fixed throughout)
\begin{eqnarray*}
\E|\varphi_N(\theta)
-\varphi(\theta)|\le \int_{0}^{\theta}\!\!
\E\left|\partial_\theta \varphi_N(\theta') 
-\partial_\theta \varphi(\theta')\right|\, \de\theta'\, .
\end{eqnarray*}
We will then show that
$\left|\partial_\theta \varphi_N(\theta,\lambda) 
-\partial_\theta \varphi(\theta,\lambda)\right|$
is bounded by $(k+1)/(1-\theta^2)$, and is smaller than $\ve$ whp for any $\ve>0$.
This implies the thesis by applying dominated convergence theorem to the
above integral.

An elementary calculation omitted from this abstract leads to 
$(1-\theta^2)\partial_{\theta}\varphi(\theta,\lambda)=
\frac{k+1}{2}\frac{\theta + h_*^2}{1+\theta h_*^2}$.
Analogously, simple calculus yields
$(1-\theta^2)\partial_\theta \log Z_{\theta,\lambda} = 
\sum_{(k,l) \in E}\E\{X_k X_l|G_N\}$, and therefore 
$(1-\theta^2)\partial_{\theta}\varphi_N(\theta,\lambda)=
\frac{k+1}{2}\E\{X_k X_l|G_N\}$ averaged over a uniformly random edge 
$(k,l)\in E$. As a consequence we have
$|\partial_{\theta}\varphi|,|\partial_{\theta}\varphi_N|\le
(k+1)/2(1-\theta^2)$ and, because of Lemma \ref{lemma:Energy}
$|\partial_{\theta}\varphi-\partial_{\theta}\varphi_N|\le\ve$
whp. This proves our claim.
\end{proof}
\begin{proof}[Proof of Theorem \ref{thm:ferromagnet}]
Throughout the proof, we set $\lambda=0$. Let  
us first prove that $k\theta \le 1$ reconstruction is unsolvable, i.e.
for any $\ve>0$ there exists $t$ such that 
$||\prob_{r,\cBall(r,t)}\{\,\cdot\,,\,\cdot\,|G_N\}-\prob_{r}\{\,\cdot\,|G_N\}\prob_{\cBall(r,t)}\{\,\cdot\,|G_N\}||_{\sTV}\le \ve$ whp.
By the Markov property of $\prob\{\,\cdot\,|G_N\}$,
(and using shorthands $\cBall$ for $\cBall(r,t)$ and $\Edge$ for $\Edge(r,t)$)
\begin{align}
||\prob_{r,\cBall}&\{\,\cdot\,,\,\cdot\,|G_N\}
-\prob_{r}\{\,\cdot\,|G_N\}\prob_{\cBall}\{\,\cdot\,|G_N\}||_{\sTV}\le
\label{eq:UniformBound}\\
&\le \sup_{\ux_{\Edge}} ||\prob_{i|\Edge}
\{\,\cdot\,|\uX_{\Edge}=\ux_{\Edge},G_N\}-\prob_{r}\{\,\cdot\,|G_N\}||_{\sTV}\, .\nonumber
\end{align}
By symmetry of $\prob\{\,\cdot\,|G_N\}$ under exchange of $+1$ and $-1$
at $\lambda=0$, cf. Eq.~(\ref{eq:IsingExplicit}), $\prob_{r}\{+1|G_N\}=1/2$.
On the other hand, by Griffiths inequalities, the probability for $X_r=+1$
is a monotone function of the values other spins are conditioned to.
Therefore the right hand side of Eq.~(\ref{eq:UniformBound}) equals
$|\prob_{r|\Edge}\{+1|\uX_{\Edge}=\underline{+1},\Ball\}-1/2|$
(we emphasized that, conditional on $\uX_{\Edge}$,
$X_r$ depends on $G_N$ only through $\Ball$). 

Finally, we recall that $\Ball(r,t)$ is with high probability a $k+1$
regular tree of depth $t$ rooted at $r$. Assuming this to be the case,
the conditional distribution of the root variable can be 
computed through a recursive dynamic-programming procedure, that we omit.
The result is (for $h_{+,s}\equiv f^{\circ s}_{\theta,0}(+1)$)
\begin{eqnarray*}
\prob_{i|\Edge}\{+1|\uX_{\Edge}=\underline{+1},\Ball\} = \frac{1}{2}
\Big\{1+g_{\theta,0}(h_{+,t-1})\Big\}\, ,
\end{eqnarray*}
where $g_{\theta,\lambda}(h)$ is defined as $f_{\theta,\lambda}(h)$,
cf. Eq.~(\ref{eq:fdef}), with $k$ replaced by $k+1$.
The thesis follows from Lemma \ref{lemma:fproperties}.

We shall now prove that non-reconstructibility implies  $k\theta\le 1$.
If $\oX\equiv N^{-1}\sum_iX_i$ is the `magnetization,' 
we claim that  non-reconstructibility implies $|\oX|\le \delta$ whp for any
$\delta>0$. In fact, because of 
non-reconstructibility, we can fix $t$ in such a way that 
$||\prob_{r,\cBall(r,t)}\{\,\cdot\,,\,\cdot\,|G_N\}-
  \prob_{r}\{\,\cdot\,|G_N\}\prob_{\cBall(r,t)}\{\,\cdot\,,|G_N\}||_{\sTV}
\le\delta$ whp.
Further notice that $\E\{\oX^2\}=\E\{X_rX_j\}$
for two uniformly random vertex $r,j\in[N]$.
Since $r\notin \Ball(i,t)$ whp and $\E\{X_r|G_N\}=0$ by symmetry, we then 
have $\E\{X_rX_j|G_N\}\le \delta$ whp, and, as a consequence, 
$\E\{\oX^2\}\le 2\delta$ for all $N$ large enough.
The claim follows from this result together with $|\oX|\le 1$

The thesis is proved by contradiction showing that for 
$k\theta >1$, there is some $\delta>0$ such that $|\oX|>\delta$
whp. Denote by $\Typ$ the set of graphs $G_N$ such that 
$Z_{\theta,0}\ge e^{N[\varphi(\theta,0)-\xi]}$ for some 
$\xi$ to be fixed below. Then, for any $G_N\in \Typ$, by 
Eq.~(\ref{eq:MagnPartFun}) and discussion below,  
\begin{eqnarray*}
\prob\{|\oX|\le \delta|G_N\}\le e^{-N[\varphi(\theta,0)-\xi]}\,
\sum_{|M|\le \delta} \Zh_{\theta,M}\, .
\end{eqnarray*}
Since by Lemma \ref{lmm:Ztotal}, $\prob\{|\oX|\le \delta\}\le
\E\{\prob\{|\oX|\le \delta|G_N\}\,\ind_{G_N\in\Typ}\}+o_N(1)$, we 
then have (estimating $\E\Zh_{\theta,M}$ with Lemma 
\ref{lmm:Zmagnetized},
\begin{eqnarray*}
\prob\{|\oX|\le \delta\}\le N\,
e^{-N[\varphi(\theta,0)-\xi]}\times C2^{N}\,e^{ND\delta^2}+o_N(1)\, .
\end{eqnarray*}
The proof follows by showing that $\varphi(\theta,0)>\log 2$,
and taking $\delta$ and $\xi$ small enough to make the first term above 
exponentially small as $N\to\infty$.

To show that  $\varphi(\theta,0)>\log 2$ for $k\theta>1$, observe that
$\phi(\theta,0,0) = \log 2$ and (after some calculus)
\begin{eqnarray*}
\left.\frac{\partial\phi}{\partial h}\right|_{\lambda=0}=-
\frac{(k+1)\theta h}{1+\theta h^2}+\frac{(k+1)
\theta f_{\theta,0}(h)}{1+\theta h f_{\theta,0}(h)}\, .
\end{eqnarray*}
Since $f_{\theta,0}(0)=k\theta h+O(h^2)$, $h=0$ is 
a local minimum (at $\theta$ fixed) of $\phi(\theta,0,h)$.
Further $\partial_{h}\phi=0$ if and only if $h=f_{\theta,\lambda}(h)$.
Recall \cite{Martinelli} that, for, $k\theta>1$, $f_{\theta,\lambda}$ has 
3 fixed points: $h=0$ and $h=\pm h_*$, for $h_*>0$. 
As a consequence $h_*$ must be a local maximum and hence
$\varphi(\theta,0)=\phi(\theta,0,h_*)>\phi(\theta,0,0) = \log 2$.
\end{proof}
\vspace{0.cm}

{\bf  Acknowledgments}
\vspace{0.2cm}

This work was largely motivated by a discussion of one of the
authors (A.M.) with Elchanan Mossel, and by the desire to substantiate 
the claims made in that discussion. It is a pleasure to acknowledge 
the debt we have with him.
%
%
%
\bibliographystyle{amsalpha}

\end{document}